\documentclass[12pt]{amsart}
\usepackage{amsfonts,amssymb,amscd,amsmath,enumerate,verbatim, fancyvrb}
\usepackage[latin1]{inputenc}
\usepackage{amscd}
\usepackage{latexsym}
\usepackage{mathptmx}
\usepackage{multicol}
\input xy
\xyoption{all}

\def\NZQ{\mathbb}               

\def\ZZ{{\NZQ Z}}

\def\frk{\mathfrak}               

\def\Phi{{\frk N}}
\def\xb{{\bold x}}
\def\yb{{\bold y}}

\def\opn#1#2{\def#1{\operatorname{#2}}} 
\opn\chara{char} 
\opn\length{\ell} 
\opn\pd{pd} 
\opn\rk{rk}
\opn\projdim{proj\,dim} 
\opn\injdim{inj\,dim} 
\opn\rank{rank}
\opn\depth{depth} 
\opn\grade{grade} 
\opn\height{height}
\opn\embdim{emb\,dim} 
\opn\codim{codim}

\opn\Tr{Tr} 
\opn\bigrank{big\,rank}
\opn\superheight{superheight}
\opn\lcm{lcm}
\opn\trdeg{tr\,deg}
\opn\reg{reg} 
\opn\lreg{lreg} 
\opn\ini{in} 
\opn\lpd{lpd}
\opn\size{size}
\opn\mult{mult}
\opn\dist{dist}
\opn\cone{cone}
\opn\lex{lex}
\opn\rev{rev}
\opn\div{div} \opn\Div{Div} \opn\cl{cl} \opn\Cl{Cl}
\opn\Spec{Spec} \opn\Supp{Supp} \opn\supp{supp} \opn\Sing{Sing}
\opn\Ass{Ass} \opn\Min{Min}
\opn\Ann{Ann} \opn\Rad{Rad} \opn\Soc{Soc}
\opn\Syz{Syz} \opn\Im{Im} \opn\Ker{Ker} \opn\Coker{Coker}
\opn\Am{Am} \opn\Hom{Hom} \opn\Tor{Tor} \opn\Ext{Ext}
\opn\End{End} \opn\Aut{Aut} \opn\id{id} \opn\ini{in}

\opn\nat{nat}
\opn\pff{pf}
\opn\Pf{Pf} \opn\GL{GL} \opn\SL{SL} \opn\mod{mod} \opn\ord{ord}
\opn\Gin{Gin}
\opn\Hilb{Hilb}\opn\adeg{adeg}\opn\std{std}\opn\ip{infpt}
\opn\Pol{Pol}
\opn\sat{sat}
\opn\Var{Var}
\opn\Gen{Gen}

\opn\aff{aff} \opn\con{conv} \opn\relint{relint} \opn\st{st}
\opn\lk{lk} \opn\cn{cn} \opn\core{core} \opn\vol{vol}
\opn\link{link} \opn\star{star}
\opn\gr{gr}

\def\pot#1#2{#1[\kern-0.28ex[#2]\kern-0.28ex]}

\opn\dirlim{\underrightarrow{\lim}}
\opn\inivlim{\underleftarrow{\lim}}
\def\Implies{\ifmmode\Longrightarrow \else
        \unskip${}\Longrightarrow{}$\ignorespaces\fi}
\def\implies{\ifmmode\Rightarrow \else
        \unskip${}\Rightarrow{}$\ignorespaces\fi}
\def\iff{\ifmmode\Longleftrightarrow \else
        \unskip${}\Longleftrightarrow{}$\ignorespaces\fi}

\let\:=\colon
\newtheorem{Theorem}{Theorem}[section]
\newtheorem{Lemma}[Theorem]{Lemma}

\newtheorem{Proposition}[Theorem]{Proposition}

\newtheorem{Example}[Theorem]{Example}

\let\epsilon\varepsilon
\let\phi=\varphi
\let\kappa=\varkappa
\def\qed{\ifhmode\textqed\fi
      \ifmmode\ifinner\quad\qedsymbol\else\dispqed\fi\fi}
\def\textqed{\unskip\nobreak\penalty50
       \hskip2em\hbox{}\nobreak\hfil\qedsymbol
       \parfillskip=0pt \finalhyphendemerits=0}
\def\dispqed{\rlap{\qquad\qedsymbol}}

\opn\dis{dis}
\opn\height{height}
\opn\dist{dist}
\def\pnt{{\raise0.5mm\hbox{\large\bf.}}}

\opn\Lex{Lex}

\begin{document}
\title{Regularity and h-polynomials of binomial edge ideals}
\author{Takayuki Hibi and Kazunori Matsuda}
\address{Takayuki Hibi,
Department of Pure and Applied Mathematics,
Graduate School of Information Science and Technology,
Osaka University, Suita, Osaka 565-0871, Japan}
\email{hibi@math.sci.osaka-u.ac.jp}
\address{Kazunori Matsuda,
Kitami Institute of Technology, 
Kitami, Hokkaido 090-8507, Japan}
\email{kaz-matsuda@mail.kitami-it.ac.jp}
\begin{abstract}
Let $G$ be a finite simple graph on the vertex set $[n] = \{ 1, \ldots, n \}$ and $K[\xb, \yb] = K[x_1, \ldots, x_n, y_1, \ldots, y_n]$ the polynomial ring in $2n$ variables over a field $K$ with each $\deg x_i = \deg y_j = 1$.  The binomial edge ideal of $G$ is the binomial ideal $J_G \subset K[\xb, \yb]$ which is generated by those binomials $x_iy_j - x_jy_i$ for which $\{i,j\}$ is an edge of $G$.  The Hilbert series $H_{K[\xb, \yb]/J_G}(\lambda)$ of $K[\xb, \yb]/J_G$ is of the form $H_{K[\xb, \yb]/J_G}(\lambda) = h_{K[\xb, \yb]/J_G}(\lambda)/(1 - \lambda)^d$, where $d = \dim K[\xb, \yb]/J_G$ and where $h_{K[\xb, \yb]/J_G}(\lambda) = h_0 + h_1\lambda + h_2\lambda^2 + \cdots + h_s\lambda^s$ with each $h_i \in \ZZ$ and with $h_s \neq 0$ is the $h$-polynomial of $K[\xb, \yb]/J_G$.  It is known that, when $K[\xb, \yb]/J_G$ is Cohen--Macaulay, one has $\reg(K[\xb, \yb]/J_G) = \deg h_{K[\xb, \yb]/J_G}(\lambda)$, where $\reg(K[\xb, \yb]/J_G)$ is the (Castelnuovo--Mumford) regularity of $K[\xb, \yb]/J_G$.  
In the present paper, given arbitrary integers $r$ and $s$ with $2 \leq r \leq s$, a finite simple graph $G$ for which $\reg(K[\xb, \yb]/J_G) = r$ and $\deg h_{K[\xb, \yb]/J_G}(\lambda) = s$ will be constructed.
\end{abstract}
\maketitle

\section*{Introduction}
The binomial edge ideal of a finite simple graph was introduced in \cite {origin} and in \cite{O1} independently.  (Recall that a finite graph $G$ is {\em simple} if $G$ possesses no loop and no multiple edge.)  Let $G$ be a finite simple graph on the vertex set $[n] = \{ 1, 2, \ldots, n \}$ and $K[\xb, \yb] = K[x_1, \ldots, x_n, y_1, \ldots, y_n]$ the polynomial ring in $2n$ variables over a field $K$ with each $\deg x_i = \deg y_j = 1$.  The {\em binomial edge ideal} $J_G$ of $G$ is the binomial ideal of $K[\xb, \yb]$ which is generated by those binomials $x_i y_j - x_j y_i$ for which $\{i, j\}$ is an edge of $G$.  

Let, in general, $S = K[x_1, \ldots, x_n]$ denote the polynomial ring in $n$ variables over a field $K$ with each $\deg x_i = 1$ and $I \subset S$ a homogeneous ideal of $S$ with $\dim S/I = d$.  The Hilbert series $H_{S/I}(\lambda)$ of $S/I$ is of the form $H_{S/I}(\lambda) = (h_0 + h_1\lambda + h_2\lambda^2 + \cdots + h_s\lambda^s)/(1 - \lambda)^d$, where each $h_i \in \ZZ$ (\cite[Proposition 4.4.1]{BH}).  We say that $h_{S/I}(\lambda) = h_0 + h_1\lambda + h_2\lambda^2 + \cdots + h_s\lambda^s$ with $h_s \neq 0$ is the {\em $h$-polynomial} of $S/I$.  Let $\reg(S/I)$ denote the ({\em Castelnuovo--Mumford}\,) {\em regularity} \cite[p.~168]{BH} of $S/I$.  
It is known (e.g., \cite[Corollary B.4.1]{V}) that, when $S/I$ is Cohen--Macaulay, one has $\reg(S/I) = \deg h_{S/I}(\lambda)$. 
Furthermore, in \cite{HM1} and \cite{HM2}, for given integers $r$ and $s$ with $r, s \geq 1$, 
a monomial ideal $I$ of $S = K[x_{1}, \ldots, x_{n}]$ with $n \gg 0$ for which $\reg(S/I) = r$ and 
$\deg h_{S/I} (\lambda) = s$ was constructed.  

Let, as before, $G$ be a finite simple graph on the vertex set $[n]$ with $d = \dim K[\xb, \yb]/J_G$ and $h_{K[\xb, \yb]/J_G}(\lambda) = h_0 + h_1\lambda + h_2\lambda^2 + \cdots + h_s\lambda^s$ the $h$-polynomial of $K[\xb, \yb]/J_G$.  
%

Now, in the present paper, given arbitrary integers $r$ and $s$ with $2 \leq r \leq s$, a finite simple graph $G$ on $[n]$ with $n \gg 0$ for which $\reg(K[\xb, \yb]/J_G) = r$ and $\deg h_{K[\xb, \yb]/J_G}(\lambda) = s$ will be constructed.

\begin{Theorem}
\label{main}
Given arbitrary integers $r$ and $s$ with $2 \leq r \leq s$, there exists a finite simple graph $G$ on $[n]$ with $n \gg 0$ for which $\reg(K[\xb, \yb]/J_{G}) = r$ and $\deg h_{K[\xb, \yb]/J_{G}}(\lambda) = s$. 
\end{Theorem} 

\section{Proof of Theorem 0.1}

Our discussion starts in the computation of the regularity and the h-polynomial of the binomial edge ideal of a path graph. 

\begin{Example}
\label{path}
{\em
Let $P_{n}$ be the path on the vertex set $[n]$ with $\{1, 2\}, \{2, 3\}, \ldots, \{n - 1, n\}$ its edges.  Since $K[\xb, \yb]/J_{P_{n}}$ is a complete intersection, it follows that the Hilbert series of $K[\xb, \yb]/J_{P_{n}}$ is $H(K[\xb, \yb]/J_{P_{n}}, \lambda) = (1 + \lambda)^{n - 1}/(1 - \lambda)^{n + 1}$ and that $\reg(K[\xb, \yb]/J_{P_{n}}) = \deg h_{K[\xb, \yb]/J_{P_{n}}}(\lambda) = n - 1$.
} 
\end{Example}

Let $G$ be a finite simple graph on the vertex set $[n]$ and $E(G)$ its edge set.  The {\em suspension} of $G$ is the finite simple graph $\widehat{G}$ on the vertex set $[n+1]$ whose edge set is $E(\widehat{G}) = E(G) \cup \{ \{i, n + 1\} : i \in [n] \}$.  Given a positive integer $m \geq 2$, the {\em m-th suspension} of $G$ is the finite simple graph ${\widehat{G}}^{\,m}$ on $[n + m]$ which is defined inductively by ${\widehat{G}}^{\,m} = \widehat{{\widehat{G}}^{\,m-1}}$, where ${\widehat{G}}^{\,1} = \widehat{G}$.

\begin{Lemma}
\label{suspension}
Let $G$ be a finite connected simple graph on $[n]$ which is not complete.  Suppose that $\dim K[\xb, \yb]/J_{G} = n + 1$ and $\deg h_{K[\xb, \yb]/J_{G}} (\lambda) \geq 2$.  Then 
\begin{eqnarray*}
\reg \left(K[\xb, \yb, x_{n+1}, y_{n+1}]/J_{\widehat{G}} \right) = \reg(K[\xb, \yb]/J_{G}), \\
\deg h_{K[\xb, \yb, x_{n+1}, y_{n+1}]/J_{\widehat{G}}} (\lambda) = \deg h_{K[\xb, \yb]/J_{G}} (\lambda) + 1. 
\end{eqnarray*}
In particular, if $ \reg(K[\xb, \yb]/J_{G}) \leq \deg h_{K[\xb, \yb]/J_{G}} (\lambda)$, then 
\[
\reg \left(K[\xb, \yb, x_{n+1}, y_{n+1}]/J_{\widehat{G}} \right) < \deg h_{K[\xb, \yb]/J_{\widehat{G}}} (\lambda).
\] 
\end{Lemma}

\begin{proof}
The suspension $\widehat{G}$ is the join product (\cite[p.3]{SMK2}) of $G$ and $\{n + 1\}$, and $\widehat{G}$ is not complete.  Hence, by virtue of \cite[Theorem 2.1]{SMK1} and \cite[Theorem 2.1 (a)]{SMK2}, one has
\[
\reg \left(K[\xb, \yb, x_{n+1}, y_{n+1}]/J_{\widehat{G}} \right) = \max\{ \reg(K[\xb, \yb]/J_{G}), 2 \} = \reg(K[\xb, \yb]/J_{G}). 
\]
Furthermore, \cite[Theorem 4.6]{KS} says that 
\begin{eqnarray*}
H\left(K[\xb, \yb, x_{n+1}, y_{n+1}]/J_{\widehat{G}}, \lambda \right) &=& H\left(K[\xb, \yb]/J_{G}, \lambda \right) + \frac{2\lambda + (n - 1)\lambda^{2}}{(1 - \lambda)^{n + 2}} \\
&=& \frac{h_{K[\xb, \yb]/J_{G}} (\lambda)}{(1 - \lambda)^{n + 1}} + \frac{2\lambda + (n - 1)\lambda^{2}}{(1 - \lambda)^{n + 2}} \\ 
&=& \frac{h_{K[\xb, \yb]/J_{G}} (\lambda) \cdot (1 - \lambda) + 2\lambda + (n - 1)\lambda^{2}}{(1 - \lambda)^{n + 2}}. 
\end{eqnarray*}
Thus $\deg h_{K[\xb, \yb, x_{n+1}, y_{n+1}]/J_{\widehat{G}}} (\lambda) = \deg h_{K[\xb, \yb]/J_{G}} (\lambda) + 1$, as desired. 
\end{proof}

We are now in the position to give a proof of Theorem \ref{main}. 

\begin{proof}(Proof of Theorem \ref{main}.)  Each of the following three cases is discussed.

\smallskip

\noindent
 {\bf (Case 1)} \,  Let $2 \leq r = s$.  
Let $G = P_{r + 1}$. 
As was shown in Example \ref{path}, one has 
\[
\reg(K[\xb, \yb]/J_{G}) = \deg h_{K[\xb, \yb]/J_{G}}(\lambda) = r.
\] 

\smallskip

\noindent
{\bf (Case 2)} \,  Let $r = 2$ and $3 \leq s$.  
Let $G = K_{s-1, s-1}$ denote the complete bipartite graph on the vertex set $[2s-2]$.  
By using \cite[Theorem 1.1 (c) together with Theorem 5.4 (a)]{SZ}, one has $\reg(K[\xb, \yb]/J_{G}) = 2$ and
\begin{eqnarray*}
F(K[\xb, \yb]/J_{G}, \lambda) &=& \frac{1 + (2s - 3)\lambda}{(1 - \lambda)^{2s - 1}} + \frac{2}{(1 - \lambda)^{2s - 2}} - \frac{2\{1 + (s - 2)\lambda \}}{(1 - \lambda)^{s}} \\
&=& \frac{1 + (2s - 3)\lambda + 2(1 - \lambda) - 2\{1 + (s - 2)\lambda \}(1 - \lambda)^{s - 1}}{(1 - \lambda)^{2s - 1}}. 
\end{eqnarray*}
Hence $\deg h_{K[\xb, \yb]/J_{G}} (\lambda) = s$, as required. 

\smallskip

\noindent
{\bf (Case 3)} \,  Let $3 \leq r < s$.  Let $G = \widehat{P_{r + 1}}^{s - r}$ be the $(s - r)$-th suspension of the path $P_{r + 1}$.  Applying Lemma \ref{suspension} repeatedly shows $\reg (S/J_{G}) = r$ and 
\begin{eqnarray*}
& & H(K[\xb, \yb]/J_{G}, \lambda) \\ 
& & \\
&=& \frac{(1 + \lambda)^{r} (1 - \lambda)^{s - r} + 2\lambda \left\{ \sum_{i = 0}^{s - r - 1} (1 - \lambda)^{i} \right\} + \lambda^{2} \sum_{i = 0}^{s - r - 1} (s - 1 - i)(1 - \lambda)^{i}}{(1 - \lambda)^{s + 2}} \\
& & \\
&=& \frac{(1 + \lambda)^{r} (1 - \lambda)^{s - r} + 2\lambda \cdot \frac{1 - (1 - \lambda)^{s - r}}{\lambda} + \lambda^{2} \cdot \frac{-1 + (1 - \lambda)^{s - r} + \lambda \{ s - r(1 - \lambda)^{s - r} \}}{\lambda^{2}}}{(1 - \lambda)^{s + 2}} \\
& & \\
&=& \frac{(1 + \lambda)^{r} (1 - \lambda)^{s - r} + 2\left\{1 - (1 - \lambda)^{s - r} \right\} -1 + (1 - \lambda)^{s - r} + \lambda \{s - r(1 - \lambda)^{s - r}\} }{(1 - \lambda)^{s + 2}} \\
& & \\
&=& \frac{(1 + \lambda)^{r} (1 - \lambda)^{s - r} + 1 - (1 - \lambda)^{s - r} + \lambda \{s - r(1 - \lambda)^{s - r}\}}{(1 - \lambda)^{s + 2}} \\
& & \\
&=& \frac{1 + s\lambda + (1 - \lambda)^{s - r} \{(1 + \lambda)^{r} - 1 - r\lambda \}}{(1 - \lambda)^{s + 2}}. 
\end{eqnarray*}
Hence $\deg h_{K[\xb, \yb]/J_{G}} (\lambda) = s$, as desired. 
\end{proof}

\section{Examples}

\begin{Proposition}
\label{cycle}
The cycle $C_n$ of length $n \geq 3$ satisfies  
\[
\reg (K[\xb, \yb]/J_{C_{n}}) \leq \deg h_{K[\xb, \yb]/J_{C_{n}}} (\lambda). 
\]
\end{Proposition}

\begin{proof}
Since the length of the longest induced path of $C_{n}$ is $n - 2$,  
it follows from \cite[Theorem 1.1]{MM} and \cite[Theorem 3.2]{regularity} that $\reg (K[\xb, \yb]/J_{C_{n}}) = n - 2$.  Furthermore, \cite[Theorem 10 (b)]{ZZ} says that
\begin{equation*}
	\deg h_{K[\xb, \yb]/J_{C_{n}}} = \begin{cases} 1 & (n = 3), \\ 
								n - 1 & (n > 3).
					\end{cases}
\end{equation*}
Hence the desired inequality follows. 
\end{proof}

Let $k \geq 1$ be an integer and $p_{1}, p_{2}, \ldots, p_{k}$ a sequence of positive integers with $p_{1} \geq p_{2} \geq \cdots \geq p_{k} \geq 1$ and $p_{1} + p_{2} + \cdots + p_{k} = n$.  Let $V_{1}, V_{2}, \ldots, V_{k}$ denote a partition of $[n]$ with each $|V_{i}| = p_{i}$.  In other words, $[n] = V_{1} \sqcup V_{2} \sqcup \cdots \sqcup V_{k}$ and $V_{i} \cap V_{j} = \emptyset$ if $i \neq j$.  Suppose that 
\[
V_{i} = \left\{\sum_{j = 1}^{i - 1} p_{j} + 1, \ \sum_{j = 1}^{i - 1} p_{j} + 2, \ \ldots,  \sum_{j = 1}^{i - 1} p_{j} + p_{i} - 1,  \sum_{j = 1}^{i} p_{j} \right\}
\]
for each $1 \leq i \leq k$.  The {\em complete multipartite graph} $K_{p_{1}, \ldots, p_{k}}$ is the finite simple graph on the vertex set $[n]$ with the edge set
\[
E\left(K_{p_{1}, \ldots, p_{k}}\right) = \{ \, \{k, \ell\} \, : \, k \in V_{i}, \, \ell \in V_{j}, \, 1 \leq i < j \leq k \, \}. 
\] 

\begin{Proposition}
\label{completemultipartite}
The complete multipartite graph $G = K_{p_{1}, \ldots, p_{k}}$ satisfies 
\[
\reg (K[\xb, \yb]/J_{G}) \leq \deg h_{K[\xb, \yb]/J_{G}} (\lambda). 
\]
\end{Proposition}
\begin{proof}
We claim $\reg (K[\xb, \yb]/J_{G}) \leq \deg h_{K[\xb, \yb]/J_{G}} (\lambda)$ by induction on $k$.  If $k = 1$, then $G = K_{p_{1}}$ is the complete graph and $\reg (K[\xb, \yb]/J_{G}) = \deg h_{K[\xb, \yb]/J_{G}} (\lambda) = 1$. 

Let $k > 1$.  If $p_{k} = 1$, then $G = \widehat{G'}$, where $G' = K_{p_{1}, \ldots, p_{k - 1}}$.  Lemma \ref{suspension} as well as the induction hypothesis then guarantees that
$\reg (K[\xb, \yb]/J_{G}) \leq \deg h_{K[\xb, \yb]/J_{G}} (\lambda)$. 
Hence one can assume that $p_{k} > 1$.  In particular $G$ is not complete.  It then follows from \cite[Theorem 2.1(a)]{SMK2} that $\reg (K[\xb, \yb]/J_{G}) = 2$.  Furthermore, \cite[Corollary 4.14]{KS} says that
\begin{equation*}
	\deg h_{K[\xb, \yb]/J_{G}} (\lambda) = \begin{cases} n - p_{k} + 1 & (2p_{1} < n + 1), \\ 
								2p_{1} - p_{k} & (2p_{1} \geq n + 1).
					\end{cases} 
\end{equation*}
Since $k > 1$ and $p_{k} > 1$, one has $\deg h_{K[\xb, \yb]/J_{G}} (\lambda) \geq n - p_{k} + 1 \geq p_{1} + 1 \geq 3$. 
Thus the desired inequality follows. 
\end{proof}

Let $t \geq 3$ be an integer and $K_{1, t}$ the complete bipartite graph on $\{1, v_{1}, \ldots, v_{t}\}$ with the edge set $E(K_{1, t}) = \{ \, \{1, v_{i}\} \, : \, 1 \leq i \leq t \, \}$.  Let $p_{1}, p_{2}, \ldots, p_{t}$ be a sequence of positive integers and $P^{(i)}$ the path of length $p_{i}$ on the vertex set $\{w_{i, 1}, w_{i, 2}, \ldots, w_{i, p_{i} + 1} \}$ for each $1 \leq i \leq t$.  Then the {\em t-starlike graph} $T_{p_{1}, p_{2}, \ldots, p_{t}}$ is defined as the finite simple graph obtained by identifying $v_{i}$ with $w_{i, 1}$ for each $1 \leq i \leq t$.  Thus the vertex set of $T_{p_{1}, p_{2}, \ldots, p_{t}}$ is
\[
\{1\} \cup \bigcup_{i = 1}^{t} \{w_{i, 1}, w_{i, 2}, \ldots, w_{i, p_{i} + 1} \}
\]
and its edge set is 
\[
E(T_{p_{1}, p_{2}, \ldots, p_{t}}) = \bigcup_{i = 1}^{t} \left\{ \{ w_{i, j}, w_{i, j + 1} \} \mid 0 \leq j \leq p_{i} \right\}, 
\] 
where $w_{i, 0} = 1$ for each $1 \leq i \leq t$.  

\begin{Proposition}
\label{starlike}
The $t$-starlike graph $G = T_{p_{1}, p_{2}, \ldots, p_{t}}$ satisfies 
\[
\reg (K[\xb, \yb]/J_{G}) < \deg h_{K[\xb, \yb]/J_{G}} (\lambda).
\]
\end{Proposition}

\begin{proof}
It follows from \cite[Corollary 3.4 (2)]{JNRR} that $\reg(K[\xb, \yb]/J_{G}) = 2 + \sum_{i = 1}^{t} p_{i}$. 
Furthermore, \cite[Theorem 5.4 (a)]{SZ} guarantees that 
\begin{eqnarray*}
F(K[\xb, \yb]/J_{K_{1, t}}) &=& \frac{1}{(1 - \lambda)^{2t}} - \frac{1 + (t - 1)\lambda}{(1 - \lambda)^{t + 1}} + \frac{1 + t\lambda}{(1 - \lambda)^{t + 2}} \\
& & \\
&=& \frac{1 - \{ 1 + (t - 1) \lambda \}(1 - \lambda)^{t - 1} + (1 + t\lambda)(1 - \lambda)^{t - 2}}{(1 - \lambda)^{2t}}  \\ 
& & \\
&=& \frac{1 + (1 - \lambda)^{t - 2} \left\{ 2\lambda + (t - 1)\lambda^{2} \right\}}{(1 - \lambda)^{2t}} 
\end{eqnarray*}
Hence, by virtue of \cite[Corollary 3.3]{KS}, one has 
\[
h_{K[\xb, \yb]/J_{G}} (\lambda) = \left[1 + (1 - \lambda)^{t - 2} \left\{ 2\lambda + (t - 1)\lambda^{2} \right\} \right] \cdot (1 - \lambda)^{\sum_{i = 1}^{t} p_{i}}. 
\]
Thus 
\[
\deg h_{K[\xb, \yb]/J_{G}} (\lambda) = t + \sum_{i = 1}^{t} p_{i} > 2 + \sum_{i = 1}^{t} p_{i} = \reg(K[\xb, \yb]/J_{G}),
\] 
as required.
\end{proof}

\begin{Example}
{\em
Let $m \geq 0$ be an integer and 
$G_{m}$ the finite simple graph on the vertex set $[m + 9]$ drawn below:

\bigskip

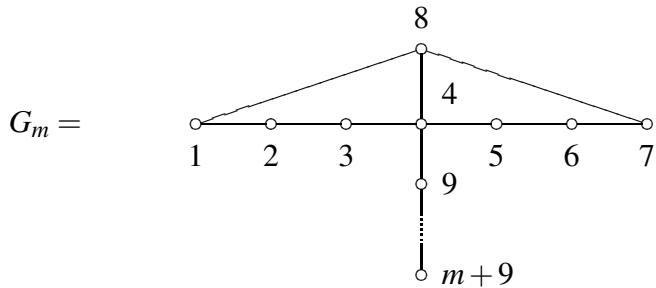
\begin{xy}
	\ar@{} (0,0);(30,0) *{\text{$G_{m} =$}};
	\ar@{} (0,0);(80, 10) *++!D{8} *\cir<2pt>{} = "A";
	\ar@{-} "A";(50, 0) *++!U{1} *\cir<2pt>{} = "B";
	\ar@{-} "B";(60, 0) *++!U{2} *\cir<2pt>{} = "C";
	\ar@{-} "C";(70, 0) *++!U{3} *\cir<2pt>{} = "D";
	\ar@{-} "D";(80, 0) *++!LD{4} *\cir<2pt>{} = "E";
	\ar@{-} "E";(90, 0) *++!U{5} *\cir<2pt>{} = "F";
	\ar@{-} "F";(100, 0) *++!U{6} *\cir<2pt>{} = "G";
	\ar@{-} "G";(110, 0) *++!U{7} *\cir<2pt>{} = "H";
	\ar@{-} "A";"H";
	\ar@{-} "A";"E";
	\ar@{-} "E";(80, -8) *++!L{9} *\cir<2pt>{} = "I";
	\ar@{-} "I";(80, -12) = "J";
	\ar@{.} "J";(80, -16) = "K";
	\ar@{-} "K";(80, -20) *++!L{m + 9} *\cir<2pt>{} = "L";
\end{xy}

\bigskip 

Then $K[\xb, \yb]/J_{G_{m}}$ is not unmixed. 
In fact, for each subset $S \subset [m + 9]$, we define 
\[
P_{S} = \left( \bigcup_{i \in S} \{ x_{i}, y_{i} \}, J_{\tilde{G}_{1}}, \ldots, J_{\tilde{G}_{c(S)}} \right), 
\]
where $G_{1}, \ldots, G_{c(S)}$ are connected components of 
$G_{[m + 9] \setminus S}$ and where $\tilde{G}_{1}, \ldots, \tilde{G}_{c(S)}$ is the 
complete graph on the vertex set $V(G_{1}), \ldots, V\left(G_{c(S)}\right)$, respectively. 
It then follows from \cite[Lemma 3.1 and Corollary 3.9]{origin}
that $P_{\emptyset}$ and $P_{\{ 3, 8 \}}$ are minimal primes of $J_{G_{m}}$ and that 
$\height P_{\emptyset} = m + 8 < \height P_{ \{3, 8\} } = m + 9$. 
Thus $K[\xb, \yb]/J_{G_{m}}$ is not unmixed.  In particular, $K[\xb, \yb]/J_{G_{m}}$ is not Cohen-Macaulay. 
However, one has 
$\reg (K[\xb, \yb]/J_{G_{m}}) = \deg h_{K[\xb, \yb]/J_{G_{m}}} (\lambda) = m + 6$.
} 
\end{Example}

\noindent
{\bf Acknowledgements} 
The first author is partially supported by JSPS KAKENHI 19H00637.  
The second author is partially supported by JSPS KAKENHI 20K03550.

\end{document}